\documentclass[12pt]{article}

\usepackage[margin=1in]{geometry}  
\usepackage{graphicx}              
\usepackage{amsmath}               
\usepackage{amsfonts}              
\usepackage{amsthm}                
\usepackage{amssymb} 
\usepackage{amsmath}
\usepackage{amstext} 

\newtheorem{thm}{Theorem}[section]
\newtheorem{lem}[thm]{Lemma}

\begin{document}

\nocite{*}

\title{An Equivalent Problem To The Twin Prime Conjecture}

\author{Francesca Balestrieri\thanks{fb340@cam.ac.uk} }

\maketitle

\begin{abstract}
	In this short paper we will show, via elementary arguments, the equivalence of the Twin Prime Conjecture
	to a problem which might be simpler to prove. Some conclusions are drawn, and it is shown that proving the Twin Prime Conjecture is equivalent to proving that there cannot be an infinite string of consecutive natural numbers satisfying some specified equations.
\end{abstract}

\section{Main Theorem}

The main theorem of this paper is the following.

\begin{thm}[Main Theorem]
  The Twin Prime Conjecture is true if, and only if, there exist infinitely many
  $n\in\mathbb{N}$ such that $n \neq 6xy + x - y$ and $n \neq 6xy + x + y$ and $n \neq 6xy - x - y$, for all
  $x, y \in \mathbb{N}$.\\
  In other words, the Conjecture is true iff $\nexists N \in \mathbb{N}$ such that $\forall n \geq N$, $n$ is of one of the forms $n = 6xy + x - y$ or $n = 6xy + x + y$ or $n = 6xy - x - y$, for some $x,y \in \mathbb{N}$.
\end{thm}

In order to prove the Main Theorem, we will need to prove some preliminary results; two thirds of this paper are devoted to this aim.

\section{Preliminary Results}
\label{sect:basics}

Consider the two sequences:
\begin{equation} a_{n} = 6n + 1 \end{equation}
\begin{equation} b_{n} = 6n - 1 \end{equation}

A simple argument shows that these two sequences generate all the prime numbers (and some other non-prime numbers). The following Lemma is a useful criterion which tells us for which $n$ the terms $a_{n}$ and $b_{n}$ are non-prime, and hence, by complement, for which $n$ the terms $a_{n}$ and $b_{n}$ are prime. \\

\begin{lem}

  Let $a_{n}$ and $b_{n}$ be the two sequences specified before. Then
  $$ a_{n} \ non-prime \Longleftrightarrow n = 6xy + x - y $$
  $$ b_{n} \ non-prime \Longleftrightarrow n = 6xy + x + y \ or \ n = 6xy + x + y $$
  for all $x, y \in \mathbb{N}$.

\end{lem}

To prove this Lemma, we need to prove some minor lemmata first.

\subsection{Some proofs}

Let $A = \{ a_{n} : n \in \mathbb{N} \}$ and let $B = \{ b_{n} : n \in \mathbb{N} \}$, where $a_{n}$ and $b_{n}$ are the sequences (1) and (2) respectively.

\begin{lem}

It is not possible to express any term $a_{k}$ of the sequence $a_{n}$ as the product $a_{x} \cdot a_{y}$ for any $(a_{x}, a_{y}) \in A \times A$, nor to express it as the product $b_{x} \cdot b_{y}$ for any $(b_{x}, b_{y}) \in B \times B$.
It is possible to express a term $a_{k}$ of the sequence $a_{n}$ as the product $a_{t} \cdot b_{r}$ of a couple
of numbers $(a_{t}, b_{r}) \in A \times B$ if, and only if, $k = 6tr + t - r$. In other words, $a_{k} = a_{t} \cdot b_{r}$ if, and only if, $k = 6tr + t - r$.

\end{lem}

\begin{proof}
Let $a_{k} = 6k - 1$, $a_{t} = 6t - 1$, $a_{x} = 6x - 1$, $a_{y} = 6y - 1$ and $b_{r} = 6r + 1$.

The last part of the theorem is almost trivial. In fact,
$$a_{t} \cdot b_{r}$$
$$= (6t - 1) \cdot (6r + 1)$$
$$= 36tr - 6r + 6t - 1$$
$$= 6 (6tr - r + t) - 1$$
and hence $a_{k} = a_{t} \cdot b_{r}$ if, and only if, $k = 6tr + t - r$ .

For the first part, consider
$$ a_{x} \cdot  a_{y}$$
$$= (6x - 1) \cdot (6y - 1)$$
$$= 36xy - 6x - 6y + 1$$
$$= 36xy - 6x - 6y + 2 - 1$$
$$= 2 (18xy - 3x - 3y + 1) - 1.$$

Hence, we must show that $18xy - 3x - 3y + 1$ is not divisible by 3. This is easy, since
$$18xy - 3x - 3y + 1 = 3 (9xy - x - y) + 1 \equiv 1 \ (mod \ 3).$$
Therefore, there are no $x, y \in \mathbb{N}$ such that $a_{x} \cdot a_{y} = a_{k}$ .

Similarly, consider
$$b_{x} \cdot  b_{y}$$
$$= (6x + 1) \cdot (6y + 1)$$
$$= 36xy + 6x + 6y + 1$$
$$= 36xy + 6x + 6y + 2 - 1$$
$$= 2 (18xy + 3x + 3y + 1) - 1.$$
Hence, we must show that $18xy + 3x + 3y + 1$ is not divisible by 3. Again, this is
straightforward since
$$18xy + 3x + 3y + 1 = 3 (9xy + x + y) + 1 \equiv 1 \ (mod \ 3).$$
Therefore, there are no $x, y \in \mathbb{N}$ such that $b_{x} \cdot b_{y} = a_{k}$ .
\end{proof}

A similar lemma can be proved for the terms of the sequence $b_{n}$.

\begin{lem}

It is not possible to express any term $b_{k}$ of the sequence $b_{n}$ as the product $a_{x} \cdot d_{y}$ for any $(a_{x}, d_{y}) \in A \times B$.
It is possible to express a term $b_{k}$ of the sequence $b_{n}$ as the product $a_{t} \cdot a_{r}$ of a couple of numbers $(a_{t}, a_{r}) \in A \times A$ if, and only if, $k = 6tr - t - r$, and as the product $b_{t} \cdot b_{r}$ of a couple of numbers $(b_{t}, b_{r}) \in B \times B$ if, and only if, $k = 6tr + t + r$. In other words, $b_{k} = a_{t} \cdot a_{r}$ if, and only if, $k = 6tr - t - r$ and $b_{k} = b_{t} \cdot b_{r}$ if, and only if, $k = 6tr + t + r$.

\end{lem}

\begin{proof}
Let $b_{k} = 6k + 1$, $a_{t} = 6t - 1$, $a_{r} = 6r - 1$, $a_{x} = 6x - 1$, $b_{t} = 6t + 1$, $b_{r} = 6r + 1$ and $b_{y} = 6y + 1$.

The last part of the theorem is almost trivial. In fact,
$$a_{t} \cdot a_{r}$$
$$= (6t - 1) \cdot (6r - 1)$$
$$= 36tr - 6r - 6t + 1$$
$$= 6 (6tr - r - t) + 1$$
and hence $b_{k} = a_{t} \cdot a_{r}$ if, and only if, $k = 6tr - t - r$ .

Also,
$$b_{t} \cdot b_{r}$$
$$= (6t + 1) \cdot (6r + 1)$$
$$= 36tr + 6r + 6t + 1$$
$$= 6 (6tr + r + t) + 1$$ and hence $b_{k} = b_{t} \cdot b_{r}$ if, and only if, $k = 6tr + t + r$ .

For the first part, consider
$$ a_{x} \cdot  b_{y}$$
$$= (6x - 1) \cdot (6y + 1)$$
$$= 36xy + 6x - 6y - 1$$
$$= 36xy + 6x - 6y - 2 + 1$$
$$= 2 (18xy + 3x - 3y - 1) + 1.$$

Hence, we must show that $18xy + 3x - 3y - 1$ is not divisible by 3. This is easy, since
$$18xy + 3x - 3y - 1 = 3 (9xy + x - y) - 1 \equiv -1 \ (mod \ 3).$$
Therefore, there are no $x, y \in \mathbb{N}$ such that $a_{x} \cdot b_{y} = b_{k}$ .
\end{proof}

The following is an obvious result.

\begin{lem} Given any term $a_{k}$ of the sequence $a_{n}$, all the primes smaller than $a_{k}$ have
already been generated by the sequences $a_{n}$ and $b_{n}$ for $n < k$ . Similarly, given any
term $b_{k}$ of the sequence $b_{n}$, all the primes smaller than $b_{k}$ have already been
generated by the sequences $a_{n}$ and $b_{n}$ for $n \leq k$.
\end{lem}
\begin{proof}
Consider the sequence $p_{n} = (a_{1}$ , $b_{1}$ , $a_{2}$ , $b_{2}$, $a_{3}$ , $b_{3}$ , \ldots , $a_{k}$ , $b_{k}$, $\ldots )$. Evidently, $p_{n}$ is a strictly increasing sequence. Furthermore, $p_{n}$ contains all the prime numbers.
Suppose that one prime number $p$ smaller than $a_{k}$ is not generated before the term $a_{k}$; then, since all the prime numbers are generated by the sequence $p_{n}$, $p$ must be generated after $a_{k}$. But the sequence $p_{n}$ is strictly increasing, and therefore $p$ is greater than $a_{k}$. This is a contradiction, and hence all the primes smaller than $a_{k}$ are generated before $a_{k}$.\\

The second half of the Lemma can be proved in a similar way.
\end{proof}

\subsection{More proofs}

A few remarks and observations are now necessary.\\

\subsubsection{}

Given a term $a_{k}$, we have shown that $$a_{k} = a_{t} \cdot b_{r} \Longleftrightarrow k = 6tr + t - r ;$$ otherwise, $a_{k}$ cannot be expressed as the product of any other pair of terms in $A \times A$, $B \times B$ or $A \times B$.\\

This can be generalised.\\

\textit{Note: An improper but self-evident use of notation will be made in the next paragraph.}\\

Let $a_{k}$ be a non-prime term of the sequence $a_{n}$. Using Lemma 2.4, all the factors of $a_{k}$
will be terms $a_{i}$ , $b_{j}$ for some $i, j < k$.
Let say that $\alpha$ factors of $a_{k}$ belong to the set $A$, and $\beta$ factors of $a_{k}$ belong to the set $B$.
Since $a_{k}$ is non-prime, $\alpha + \beta > 1$. In short form,
$$ a_{k} = A^{\alpha} \cdot B^{\beta} \ \  (with \  \alpha + \beta > 1),$$
where $A$ (improperly) denotes an element of $A$ and $B$ (improperly) denotes an element
of $B$.\\

There are a few cases to consider, depending on the values of $\alpha$ and $\beta$.
Notice that in each case we will repetively use Lemma 2.2 and Lemma 2.3.\\

\begin{enumerate}
	\item
If $\alpha = 0$, then for any natural value of $\beta \ (>1)$, we will have that $a_{k} = B^{\beta}$. But $B^{\beta}$ is
an element of the set $B$. Hence, we have an equation with an element of the set $A$ in
the LHS and an element of the set $B$ in the RHS, and $A \cap B = \emptyset$. This is nonsense, and hence it cannot be that $\alpha = 0$.

\item
If $\alpha$ is even, then for any $\beta \in \mathbb{N}\cup\{0\}$, we can write $a_{k}$ as
$$a_{k} = A^{\alpha} \cdot B^{\beta} = (A \cdot A) \cdot (A \cdot A) \cdot \ldots \cdot (A \cdot A) \cdot B^{\beta} = B^{\alpha/2} \cdot B^{\beta} = B^{(\alpha/2) + \beta}$$
and we have a situation analogous to the one in case 1. Hence, $\alpha$ cannot be even.

\item
If $\alpha$ is odd, then $\alpha - 1$ is positive even or zero) and, for any $\beta \in \mathbb{N} \cup \{0\}$, we can write $a_{k}$ as
$$ a_{k} = A^{\alpha} \cdot B^{\beta} = A \cdot A^{\alpha - 1} \cdot B^{\beta} = A \cdot B^{(\alpha - 1)/2} \cdot B^{\beta} = A \cdot B^{(\alpha - 1)/2 + \beta}$$
and this is consistent with what we have already proved, since $A$ denotes an element $a_{t}$ of the set $A$ and $B^{(\alpha - 1)/2 + \beta}$ is an element $b_{r}$ of $B$, that is, $A \cdot B^{(\alpha - 1)/2 + \beta} = a_{t} \cdot b_{r}$ for some $t,
r \in \mathbb{N}$.

So, if $a_{k}$ is non-prime, it is necessarely of the form $a_{k} = A^{\alpha} \cdot B^{\beta}$ , where $\alpha$ is odd, $\beta$ is any number in $\mathbb{N} \cup \{0\}$, and $(\alpha + \beta) > 1$. Of course, the converse is also true.

Hence, $a_{k}$ is non-prime $\Longleftrightarrow  a_{k} = A^{\alpha} \cdot B^{\beta}$  where $\alpha$ is odd, $\beta$ is any number in $\mathbb{N} \cup \{0\}$, and $(\alpha + \beta) > 1 \Longleftrightarrow  a_{k} = a_{t} \cdot b_{r}$ for some $t, r \in \mathbb{N} \Longleftrightarrow k = 6tr + t - r$.\\

By Lemma 2.4, in all the other cases $a_{k}$ is prime.

\end{enumerate}

\subsubsection{}

Given a term $b_{k}$, we have shown that
$$b_{k} = a_{t} \cdot a_{r} \Longleftrightarrow k = 6tr - t - r ;$$
$$b_{k} = b_{t} \cdot b_{r} \Longleftrightarrow k = 6tr + t + r ;$$ otherwise, $b_{k}$ cannot be expressed as the product of any other pair of terms in $A \times A$, $B \times B$ or $A \times B$.\\

This, again, can be generalised.\\

\textit{Note: An improper but self-evident use of notation will be made in the next paragraph.}\\

Let $b_{k}$ be a non-prime term of the sequence $b_{n}$. Using Lemma 2.4, all the factors of $b_{k}$
will be terms $a_{i}$ , $b_{j}$ for some $i \leq k$ and some $j < k$.
Let say that $\alpha$ factors of $b_{k}$ belong to the set $A$, and $\beta$ factors of $b_{k}$ belong to the set $B$.
Since $b_{k}$ is non-prime, $\alpha + \beta > 1$. In short form,
$$ b_{k} = A^{\alpha} \cdot B^{\beta} \ \  (with \  \alpha + \beta > 1),$$
where $A$ (improperly) denotes an element of $A$ and $B$ (improperly) denotes an element
of $B$.

There are a few cases to consider, depending on the values of $\alpha$ and $\beta$.
Notice that in each case we will repetively use Lemma 2.2 and Lemma 2.3.\\

\begin{enumerate}
	\item
If $\alpha = 0$, then for any natural value of $\beta \ (>1)$, we will have that $b_{k} = B^{\beta}$, and this is
consistent with what we have already proved, since $B$ denotes an element $b_{t}$ of the set $B$ and $B^{\beta - 1}$ is an element $b_{r}$ of $B$, that is, $B \cdot B^{\beta - 1} = b_{t} \cdot b_{r}$ for some $t, r \in \mathbb{N}$.

\item
If $\alpha$ is even, then for any $\beta \in \mathbb{N}\cup\{0\}$, we can write $b_{k}$ as
$$b_{k} = A^{\alpha} \cdot B^{\beta} = (A \cdot A) \cdot (A \cdot A) \cdot \ldots \cdot (A \cdot A) \cdot B^{\beta} = B^{\alpha/2} \cdot B^{\beta} = B^{(\alpha/2) + \beta}$$
and we have a situation analogous to the one in case 1, which is consistent. \\
Furthermore, if $\alpha$ is even, then $\alpha - 2$ is positive even or zero. Then for any $\beta \in \mathbb{N}\cup\{0\}$, we can write $b_{k}$ as
$$b_{k} = A^{\alpha} \cdot B^{\beta} = A \cdot A \cdot A^{\alpha - 2} \cdot B^{\beta} = A \cdot A \cdot B^{(\alpha - 2)/2} \cdot B^{\beta} =  (A \cdot B^{(\alpha - 2)/2}) \cdot (A \cdot B^{\beta})$$
and this is consistent with what we have already proved, since $(A \cdot B^{(\alpha - 2)/2})$ is an
element $a_{t}$ of the set $A$ and $(A \cdot B^{\beta})$ is an element $a_{r}$ of the set $A$ , that is, $(A \cdot B^{(\alpha - 2)/2}) \cdot (A \cdot B^{\beta}) = a_{t} \cdot a_{r}$ for some $t, r \in \mathbb{N}$.

\item
If $\alpha$ is odd, then $\alpha - 1$ is positive even or zero and, for any $\beta \in \mathbb{N} \cup \{0\}$, we can write $b_{k}$ as
$$ b_{k} = A^{\alpha} \cdot B^{\beta} = A \cdot A^{\alpha - 1} \cdot B^{\beta} = A \cdot B^{(\alpha - 1)/2} \cdot B^{\beta} = A \cdot B^{(\alpha - 1)/2 + \beta}.$$
But $A \cdot B^{(\alpha - 1)/2 + \beta}$ is an element of the set $A$, and hence we have an equation with an element of the set $B$ in the LHS and an element of the set $A$ in the RHS. Since, $A \cap B = \emptyset$, this is nonsense, and therefore it cannot be that $\alpha$ is odd.\\

\end{enumerate}

So, if $b_{k}$ is non-prime, it is necessarely of the form $b_{k} = A^{\alpha} \cdot B^{\beta}$ , where $\alpha$ is positive even or zero, $\beta$ is any number in $\mathbb{N} \cup \{0\}$, and $(\alpha + \beta) > 1$. Of course, the converse is also true.\\

Hence, $b_{k}$ is non-prime $\Longleftrightarrow  a_{k} = A^{\alpha} \cdot B^{\beta}$  where $\alpha$ is even positive or zero, $\beta$ is any number in $\mathbb{N} \cup \{0\}$, and $(\alpha + \beta) > 1 \Longleftrightarrow  b_{k} = b_{t} \cdot b_{r}$ or $b_{k} = a_{t} \cdot a_{r}$  for some $t, r \in \mathbb{N} \Longleftrightarrow k = 6tr + t + r$ or $k = 6tr - t - r$ .\\

By Lemma 2.4, in all the other cases $b_{k}$ is prime.

\section{Proof of the Main Theorem}

We are now ready to give a proof of the Main Theorem.

\begin{thm}[Main Theorem]
  The Twin Prime Conjecture is true if, and only if, there exist infinitely many
  $n\in\mathbb{N}$ such that $n \neq 6xy + x - y$ and $n \neq 6xy + x + y$ and $n \neq 6xy - x - y$, for all
  $x, y \in \mathbb{N}$.
\end{thm}

\begin{proof}
	Let $p$ be prime, with $p \geq 5$. Clearly if $(p, p+2)$ is a twin primes couple, then we must have that $p$ belongs to
	the sequence $a_{n}$ and $p + 2$ belongs to the sequence $b_{n}$. That is, $ p = a_{n}= 6n - 1$ and $p + 2 = b_{n} = 6n + 1$, for a same $n \in \mathbb{N}$. But by Lemma 2.1, $(a_{n}, b_{n})$ is a twin primes couple if, and only if, $n \neq 6xy + x - y$ and $n \neq 6xy + x + y$ and $n \neq 6xy - x - y$ for all $x, y \in \mathbb{N}$. Hence, there are infinitely many twin primes couples $(a_{n}, b_{n})$ (and thus infinitely many twin primes) if, and only if, there exist infinitely many $n \in \mathbb{N}$ such that $n \neq 6xy + x - y$ and $n \neq 6xy + x + y$ and $n \neq 6xy - x - y$ for all $x, y \in \mathbb{N}$. This completes the proof.
 \end{proof}

\section{Considerations}
\label{subsect:typing}

In this new form, the Twin Prime Conjecture seems to suggest a natural way of proving it which a reductio ad absurdum type of argument.\\ 

Note that the Lemma 2.1 proves that, for $p$ prime such that $p \geq 5$, the twin primes couples $(p, p+2)$ are exactly those $(a_{n}, b_{n}) = (6n -1, 6n +1)$ for which $n \neq 6xy + x - y$ and $n \neq 6xy + x + y$ and $n \neq 6xy - x - y$ for all $x, y \in \mathbb{N}$. Thus, Lemma 2.1 gives a way to find twin primes: it is sufficient to find an $n$ satisfying the above conditions to get a twin primes couple $(6n -1, 6n +1)$.\\

Assuming the negation of the Conjecture would mean assuming the existence of an infinite string of consecutive natural numbers such that, for any $n$ in this string, $n = 6xy + x - y$ or $n = 6xy + x + y$ or $n = 6xy - x - y$ for some $x, y \in \mathbb{N}$; if this could lead to a contradiction, then the Conjecture would prove to be true. (Nevertheless, a direct proof of the impossibility of constructing such an infinite string would be equally effective - if we want to prove the truth of the Conjecture, of course.)
\newline


\section{References}

[1] Hardy, G.H. and Wright, E.M.; (1988) \emph{An Introduction to the Theory of Numbers}, Clarendon Press, Oxford, 5th ed.

\end{document}